\documentclass[12pt,reqno]{amsart}

\usepackage{a4}
\usepackage{enumerate}
\usepackage{hyperref}

\DeclareMathOperator{\eend}{end}
\DeclareMathOperator{\grad}{grad}

\theoremstyle{plain}
\newtheorem*{theorem*}{Theorem}
\newtheorem*{conjecture*}{Conjecture}

\newcommand\itemref[1]{(\ref{#1})}
\newcommand\llangle{\langle\!\langle}
\newcommand\rrangle{\rangle\!\rangle}

\begin{document}

\title[Analytic eigenbranches]{Analytic eigenbranches\\ in the semi-classical limit}

\author{Stefan Haller}

\address{Stefan Haller,
         Department of Mathematics,
         University of Vienna,
         Os\-kar-Mor\-gen\-stern-Platz 1,
         1090 Vienna,
         Austria.}

\email{stefan.haller@univie.ac.at}

\begin{abstract}
We consider a one parameter family of Laplacians on a closed manifold and study the semi-classical limit of its analytically parametrized eigenvalues.
Our results are analogous to a theorem for scalar Schr\"odinger operators on Euclidean space by Luc Hillairet and apply to geometric operators like Witten's Laplacian associated with a Morse function.
\end{abstract}



\maketitle

Let $M$ be a closed smooth manifold and suppose $E$ is a complex vector bundle over $M$.
Fix a smooth volume density on $M$ and a smooth fiber wise Hermitian metric on $E$.
We will denote the associated $L^2$ scalar product on the space of sections, $\Gamma(E)$, by $\llangle-,-\rrangle$ and write $\|-\|$ for the corresponding $L_2$ norm.

Consider a one parameter family of operators acting on $\Gamma(E)$,
\begin{equation}\label{E:1}
\Delta_t=\Delta+tA+t^2V,\qquad t\in\mathbb R,
\end{equation}
where $\Delta$ is a selfadjoint (bounded from below) Laplacian and $A,V\in\Gamma(\eend(E))$ are smooth symmetric sections. 
This is a selfadjoint holomorphic family of type (A) in the sense of \cite[Section~VII\S2]{K66}.
According to the Kato--Rellich theorem, see \cite[Theorem~VII.3.9]{K66}, its eigenvalues can be organized in analytic families
referred to as \emph{analytic eigenbranches} of $\Delta_t$.
More precisely, there exist eigenbranches $\lambda_t$ and eigensections $\psi_t$, both analytic in $t\in\mathbb R$, such that
\begin{equation}\label{E:anfam}
\Delta_t\psi_t=\lambda_t\psi_t\qquad\text{and}\qquad\|\psi_t\|=1.
\end{equation}
Furthermore, it is possible to choose a sequence of analytic eigenbranches, $\lambda^{(k)}_t$, and corresponding analytic eigensections, $\psi_t^{(k)}$, such that at every time $t$, the sequence $\lambda_t^{(k)}$ exhausts all of the spectrum of $\Delta_t$, including multiplicities, and $\psi_t^{(k)}$ forms a complete orthonormal basis of eigensections. The analytic parametrization of the spectrum, $\lambda^{(k)}_t$, is unique up to renumbering.
The eigensections, on the other hand, are by no means canonical, and it seems more natural to consider the spectral projections instead.

In this note we study the semi-classical limit of the analytic eigenbranches, i.e., the behavior of $\lambda_t$ as $t\to\infty$.
We will show that $t^{-2}\lambda_t$ converges to a finite limit $\mu$, see part \itemref{T:b} of the theorem below.
Moreover, if the potential $V$ is scalar valued, i.e., if $V=v\cdot\operatorname{id}_E$ for a smooth function $v$, then $\mu$ has to be a critical value of $v$, see part~\itemref{T:i} in the theorem below.
These observations are analogous to a result of Luc Hillairet \cite{H09} who considered the scalar case on $M=\mathbb R^n$ with $A=0$.

While Hillairet's proof uses basic properties of semi-classical measures, our proof is entirely elementary and does not make use of this concept.
The ideas entering into the proof, however, appear to be essentially the same.
Avoiding semi-classical measures makes the generalization to the vector valued case considered here straight forward.
As in Hillairet's argument, convergence of $t^{-2}\lambda_t$ follows from the fact that this quantity, suitably corrected due of the presence of $A$, is bounded and monotone, cf.~\eqref{E:mono} below.
The fact that the Laplacian is semi-bounded enters crucially at this point.

The asymptotics of the spectral distribution function in the semi-classical limit has applications in quantum mechanics and geometric topology \cite{CFKS87, DS99, HS84, HS85}.
We merely mention Witten's influential paper \cite{W82} and and proofs of the Cheeger--M\"uller theorem \cite{BZ92,BFKM96,BFK96}.
In these geometric applications, a Morse function $f$ provides a deformation of the deRham differential, $d_t=e^{-tf}de^{tf}=d+tdf$.
The corresponding Witten Laplacian, $\Delta_t=(d_t+d_t^*)^2=d_td_t^*+d_t^*d_t=[d_t,d_t^*]$, is a one parameter family of operators acting on differential forms, i.e., $E=\Lambda^*T^*M$, which is of the type considered here with scalar valued $V=|df|^2$.
Hence, in this case the absolute minima of $V$ coincide with the critical points of $f$.

Let us return to a general one parameter family of operator considered above, see~\eqref{E:1}, and an analytic eigenbranch as in \eqref{E:anfam}.
Subsequently, we will use the notation $\dot\lambda_t:=\frac\partial{\partial t}\lambda_t$, $\dot\psi_t:=\frac\partial{\partial t}\psi_t$, and 
\begin{equation}\label{E:5}
\dot\Delta_t:=\tfrac\partial{\partial t}\Delta_t=A+2tV.
\end{equation}
We have the following analogue of Theorem~1 in \cite{H09}.

\begin{theorem*}\label{T}
For each analytic eigenbranch the following hold true:
\begin{enumerate}[(a)]
\item\label{T:a}
$\dot\lambda_t=O(t)$ and $\lambda_t=O(t^2)$, as $t\to\infty$.
\item\label{T:b}
$t^{-2}\lambda_t$ converges to a finite limit, 
\begin{equation}\label{E:mu}
\mu:=\lim_{t\to\infty}t^{-2}\lambda_t.
\end{equation}
\item\label{T:c}
$t\frac\partial{\partial t}(t^{-2}\lambda_t)$ is bounded and 
$$
\limsup_{t\to\infty}t\tfrac\partial{\partial t}(t^{-2}\lambda_t)=0.
$$
\item\label{T:d}
$t^{-2}\llangle\Delta\psi_t,\psi_t\rrangle$ is bounded and 
$$
\liminf_{t\to\infty}t^{-2}\llangle\Delta\psi_t,\psi_t\rrangle=0.
$$
\item\label{T:e}
The limit, $\mu$, has the following interpretations:
$$
\mu=\lim_{t\to\infty}t^{-2}\lambda_t
=\limsup_{t\to\infty}\llangle V\psi_t,\psi_t\rrangle
=\limsup_{t\to\infty}(2t)^{-1}\dot\lambda_t
=\limsup_{t\to\infty}\tfrac\partial{\partial t}(t^{-1}\lambda_t).
$$
\end{enumerate}
Furthermore, for each sequence $t_n\to\infty$ such that, cf.~\eqref{T:d},
\begin{equation}\label{E:tn}
\lim_{n\to\infty}t_n^{-2}\llangle\Delta\psi_{t_n},\psi_{t_n}\rrangle=0,
\end{equation}
the following hold true:
\begin{enumerate}[(a)]
\setcounter{enumi}{5}
\item\label{T:f}
For every positive integer $s\in\mathbb N$,
$$
\lim_{n\to\infty}t_n^{-s}\|\psi_{t_n}\|_{H^s(M)}=0,
$$
where $\|-\|_{H^s(M)}$ denotes any Sobolev $s$ norm on $\Gamma(E)$.
\item\label{T:h}
We have
\begin{equation}\label{E:66}
\lim_{n\to\infty}\|(V-\mu)\psi_{t_n}\|=0.
\end{equation}
Hence, the eigensections $\psi_{t_n}$ localize near $\Sigma_\mu:=\{x\in M:\det(V(x)-\mu)=0\}$.
In particular, for every open neighborhood $U$ of $\Sigma_\mu$, we have 
\begin{equation}\label{E:666}
\lim_{n\to\infty}\|\psi_{t_n}\|_{L^2(M\setminus U)}=0.
\end{equation}
\item\label{T:i}
If, moreover, the potential $V$ is multiplication by a (scalar) function, then
\begin{equation}\label{E:dVpsi}
\lim_{n\to\infty}\bigl\|\, |dV|^2\psi_{t_n}\,\bigr\|=0.
\end{equation}
Hence, the eigensections $\psi_{t_n}$ localize near the critical points of $V$. For every neighborhood $\tilde U$ of the critical set of $V$ we have 
\begin{equation}\label{E:jjkl}
\lim_{n\to\infty}\|\psi_{t_n}\|_{L^2(M\setminus\tilde U)}=0.
\end{equation}
In particular, $\mu$ has to be a critical value of $V$.
\end{enumerate}
\end{theorem*}

\begin{proof}
From \eqref{E:anfam} we obtain
\begin{equation}\label{E:3}
\llangle\Delta_t\psi_t,\psi_t\rrangle=\lambda_t.
\end{equation}
Differentiating the second equation in \eqref{E:anfam} we get
$$
\llangle\dot\psi_t,\psi_t\rrangle+\llangle\psi_t,\dot\psi_t\rrangle=0.
$$
Differentiating \eqref{E:3} and using the selfadjointness of $\Delta_t$ this leads to
\begin{equation}\label{E:4}
\llangle\dot\Delta_t\psi_t,\psi_t\rrangle=\dot\lambda_t.
\end{equation}
Combining this with \eqref{E:5} we obtain
\begin{equation}\label{E:55}
\dot\lambda_t=\llangle(A+2tV)\psi_t,\psi_t\rrangle=\llangle A\psi_t,\psi_t\rrangle+2t\llangle V\psi_t,\psi_t\rrangle.
\end{equation}
Since $A$ and $V$ are bounded operators, this implies $\dot\lambda_t=O(t)$, whence \itemref{T:a}.

From \eqref{E:1} and \eqref{E:5} we immediately get
$$
2\Delta_t-t\dot\Delta_t=2\Delta+tA.
$$
Combining this with \eqref{E:3} and \eqref{E:4} and using the boundedness of $A$, we obtain
\begin{multline}\label{E:9}
\tfrac\partial{\partial t}(t^{-2}\lambda_t)
=-t^{-3}\bigl(2\lambda_t-t\dot\lambda_t\bigr)
=-t^{-3}\llangle(2\Delta_t-t\dot\Delta_t)\psi_t,\psi_t\rrangle
\\
=-t^{-3}\llangle(2\Delta+tA)\psi_t,\psi_t\rrangle
=-2t^{-3}\llangle\Delta\psi_t,\psi_t\rrangle+O(t^{-2}).
\end{multline}
Hence, as $\Delta$ is bounded from below, there exists a constant $C$ such that
\begin{equation}\label{E:mono}
\tfrac\partial{\partial t}\bigl(t^{-2}\lambda_t+Ct^{-1}\bigr)\leq0,
\end{equation}
for sufficiently large $t$.
This shows that the quantity $t^{-2}\lambda_t+Ct^{-1}$ is monotone, for large $t$.
In view of \itemref{T:a} it is bounded too. Whence $t^{-2}\lambda_t+Ct^{-1}$ converges, as $t\to\infty$.
This immediately implies \itemref{T:b}.

Similarly, one can show \itemref{T:c} and \itemref{T:d}:
Rewriting \eqref{E:9} we get
\begin{equation}\label{E:10}
-t\tfrac\partial{\partial t}(t^{-2}\lambda_t)
=2t^{-2}\llangle\Delta\psi_t,\psi_t\rrangle+O(t^{-1}).
\end{equation}
As $t^{-2}\lambda_t$ is bounded we must have
\begin{equation}\label{E:11}
\limsup_{t\to\infty}t\tfrac\partial{\partial t}(t^{-2}\lambda_t)\geq0,
\end{equation}
for otherwise $t^{-2}\lambda_t$ would diverge logarithmically.
Moreover,
\begin{equation}\label{E:12}
\liminf_{t\to\infty}t^{-2}\llangle\Delta\psi_t,\psi_t\rrangle\geq0,
\end{equation}
since $\Delta$ is bounded from below. Combining \eqref{E:10}, \eqref{E:11} and \eqref{E:12} we obtain
$$
\limsup_{t\to\infty}t\tfrac\partial{\partial t}(t^{-2}\lambda_t)=0=\liminf_{t\to\infty}t^{-2}\llangle\Delta\psi_t,\psi_t\rrangle.
$$
This completes the proof of \itemref{T:c} and \itemref{T:d}, the statements on boundedness follow immediately from \itemref{T:a} and \eqref{E:10}.

To see \itemref{T:e} note that \eqref{E:1} and \eqref{E:3} give
$$
t^{-2}\lambda_t=t^{-2}\llangle\Delta\psi_t,\psi_t\rrangle+t^{-1}\llangle A\psi_t,\psi_t\rrangle+\llangle V\psi_t,\psi_t\rrangle.
$$
Using \itemref{T:d} we obtain the second equality in \itemref{T:e}. The third equality follows from \eqref{E:55}.
The last one is immediate using $\frac\partial{\partial t}(t^{-1}\lambda_t)=t^{-1}\dot\lambda_t-t^{-2}\lambda_t$.

To see \itemref{T:f}, note first that the case $s=1$ is immediate from \eqref{E:tn} since there exists a constant $C$ such that $\|\psi\|_{H^1(M)}^2\leq C(\llangle\Delta\psi,\psi\rrangle+\|\psi\|^2)$.
Using the eigensection equation $\Delta_t\psi_t=\lambda_t\psi_t$ and \eqref{T:a}, one obtains constants $C_s$ such that 
$$
\|\psi_t\|_{H^{s+2}(M)}\leq C_s(1+t^2)\|\psi_t\|_{H^s(M)}.
$$
This permits to establish \itemref{T:f} inductively for all odd integers $s\in\mathbb N$.
The even case can be reduced to the odd one using $\|\psi\|^2_{H^2(M)}\leq C'\|\psi\|_{H^1(M)}\|\psi\|_{H^3(M)}$, an estimate which readily follows from the Cauchy--Schwarz inequality.

Using the estimate in \eqref{T:f} for $s=2$, the eigensection equation,
$$
t^{-2}\Delta\psi_t+t^{-1}A\psi_t+(V-\mu)\psi_t=(t^{-2}\lambda_t-\mu)\psi_t,
$$
implies \eqref{E:66}.
As $V-\mu$ is invertible over the compact set $M\setminus U$, there exists a constant $c>0$ such that $c^2\leq(V-\mu)^*(V-\mu)$, over each point in $M\setminus U$. Consequently, 
$$
c\,\|\psi\|_{L^2(M\setminus U)}\leq\|(V-\mu)\psi\|_{L^2(M\setminus U)}.
$$ 
Combining this with \eqref{E:66}, we obtain \eqref{E:666}.

Let us now turn to the proof of \itemref{T:i}.
Suppose $D$ is a first order differential operator acting on sections of $E$.
Since $V$ is scalar valued, the commutator $\sigma:=[D,V]=DV-VD$ is a differential operator of order zero, namely the principal symbol, $\sigma=\sigma_D(dV)\in\Gamma(\eend(E))$.
As $\Delta$ is a Laplacian, the commutator $[D,\Delta]$ is a differential operator of order at most two.
Using 
$$
[D,t^{-2}\Delta_t]=t^{-2}[D,\Delta]+t^{-1}[D,A]+\sigma,
$$ 
the Cauchy--Schwarz inequality and \itemref{T:f}, we thus obtain
\begin{equation}\label{E:jkl}
\|\sigma\psi_{t_n}\|^2=\llangle[D,t_n^{-2}\Delta_{t_n}]\psi_{t_n},\sigma\psi_{t_n}\rrangle+o(1),\qquad\text{as $n\to\infty$.}
\end{equation}
Using $\Delta_t\psi_t=\lambda_t\psi_t$ one readily checks
\begin{equation}\label{E:qwerty}
\llangle[D,t^{-2}\Delta_t]\psi_t,\sigma\psi_t\rrangle
=\llangle D\psi_t,[\sigma,t^{-2}\Delta_t]\psi_t\rrangle.
\end{equation}
Since $V$ is scalar valued we have $[\sigma,V]=0$ and 
$$
t[\sigma,t^{-2}\Delta_t]=t^{-1}[\sigma,\Delta]+[\sigma,A],
$$
where $[\sigma,\Delta]$ is a differential operator of order at most one.
Proceeding as above, the Cauchy--Schwarz inequality and \itemref{T:f} yield
$$
\lim_{n\to\infty}\llangle D\psi_{t_n},[\sigma,t_n^{-2}\Delta_{t_n}]\psi_{t_n}\rrangle=0.
$$
Combining the latter with \eqref{E:jkl} and \eqref{E:qwerty}, we arrive at
\begin{equation}\label{E:jkll}
\lim_{n\to\infty}\|\sigma\psi_{t_n}\|=0.
\end{equation}
Specializing to $D=\nabla_X$ where $\nabla$ is some linear connection on $E$ and $X$ is a vector field on $M$, we obtain $\sigma=X\cdot V=dV(X)$.
Choosing $X=\grad(V)$ with respect to some auxiliary Riemannian metric on $M$, we have $\sigma=|dV|^2$, and \eqref{E:jkll} becomes \eqref{E:dVpsi}.
On $M\setminus\tilde U$ we have $0<\tilde c\leq|dV|^2$ for some constant $\tilde c$, hence
$$
\tilde c\,\|\psi\|_{L^2(M\setminus\tilde U)}\leq \bigl\|\,|dV|^2\psi\bigr\|_{L^2(M\setminus\tilde U)},
$$
and thus \eqref{E:dVpsi} implies \eqref{E:jjkl}.
Combining the latter with \itemref{T:h}, we see that $\mu$ has to be a critical value of $V$.
This completes the proof of the theorem.
\end{proof}

\section*{Concluding remarks}

At the very end of Section~3 in \cite{H09}, Hillairet points out that for $M=\mathbb R^1$ and non-degenerate minima, the limit $\mu$ has to be the absolute minimum of $V$.
Indeed, in this situation the spectrum is known to be simple, hence the analytic eigenbranches cannot cross and remain in the same order: $\lambda_t^{(1)}<\lambda_t^{(2)}<\cdots$
The semi-classical asymptotics of the $k$-th eigenvalue, however, is governed by the deepest wells.

\begin{conjecture*}
If the `minima' of $V$ are non-degenerate in the sense of Shubin \cite[Condition~C on page~378]{S96} and $M$ is connected, then $t^{-2}\lambda_t$ converges to the absolute minimum of $V$.
\end{conjecture*}

From \eqref{E:mono} we see that there exists a constant $C$ such that 
$$
\mu-Ct^{-1}\leq t^{-2}\lambda_t
$$
for sufficiently large $t$.
It is unclear to the author, if a similar estimate from above holds true, i.e., if we have $t^{-2}\lambda_t=\mu+O(t^{-1})$ as $t\to\infty$.

As $t\to\infty$, the following statements are equivalent:
\begin{enumerate}[(a)]
\item\label{I:V} $\llangle\Delta\psi_t,\psi_t\rrangle=O(t)$
\item\label{I:D} $\llangle(V-\mu)\psi_t,\psi_t\rrangle=O(t^{-1})$
\item\label{I:alpha} $\frac\partial{\partial t}(t^{-\alpha}(\lambda_t-\mu t^2))=O(t^{-\alpha})$ for one (and then all) real $\alpha\neq1$.
\item\label{I:VV} $\frac\partial{\partial t}(t^{-2}\lambda_t)=O(t^{-2})$ 
\item\label{I:DD} $\dot\lambda_t-2t\mu=O(1)$
\end{enumerate}
Indeed, the equivalence $\itemref{I:V}\Leftrightarrow\itemref{I:VV}$ follows from \eqref{E:10}, the equivalence $\itemref{I:D}\Leftrightarrow\itemref{I:DD}$ follows from \eqref{E:55}, and the equivalence $\itemref{I:alpha}\Leftrightarrow\itemref{I:VV}\Leftrightarrow\itemref{I:DD}$ follows from \eqref{E:mu}.
Moreover, if these five (equivalent) statements hold true, then clearly
\begin{equation}\label{E:xyz}
t^{-2}\lambda_t=\mu+O(t^{-1}),\qquad\text{as $t\to\infty$.}
\end{equation}

Suppose \eqref{E:xyz} holds true.
Moreover, assume that $\mu$ is the absolute minimum of $V$ and the minima are all non-degenerate in the sense of \cite[Condition~C on page~378]{S96}.
Then, according to \cite[Theorem~1.1]{S96} or \cite[Theorem~11.1]{CFKS87}, there exists an eigenvalue $\omega$ of the harmonic oscillator associated with the minima (deepest wells) such that 
\begin{equation}\label{E:bestcase}
\lambda_t=\mu t^2+\omega t+O(t^{4/5}),\qquad\text{as $t\to\infty$.}
\end{equation}
Better estimates are available if the geometry is flat near the minima, cf.\ \cite[Equation~(1.15)]{S96}.
Under additional assumptions \cite{S83} one even obtains a full asymptotic expansion in terms of integral powers of $t$.
Furthermore, for every eigenvalue $\omega$ of the harmonic oscillator associated with the minima of $V$, there exists an analytic eigenbranch for which \eqref{E:bestcase} holds true with $\mu$ the absolute minimum of $V$ and the number of these eigenbranches coincides with the multiplicity of $\omega$.
An intriguing problem remains open: Are there analytic eigenbranches with a different asymptotics which are not governed by the deepest wells?

\section*{Acknowledgments}

The author is indebted to Dan Burghelea for encouraging discussions.
Part of this work was done while the author enjoyed the hospitality of the Max Planck Institute for Mathematics in Bonn. 
He gratefully acknowledges the support of the Austrian Science Fund (FWF): project number P31663-N35.

\end{document}